\documentclass[12pt,a4paper]{amsart} 

\title{Irreducible subshifts associated with $\tilde A_2$ buildings.}     
\date{March 4, 2002}
\author{Guyan Robertson }
\address{Mathematics Department, University of Newcastle, Callaghan, NSW 
2308, Australia}
\email{guyan@maths.newcastle.edu.au}  
\author{Tim Steger}
\address{Istituto Di Matematica e Fisica, Universit\`a degli Studi di
Sassari, Via Vienna 2, 07100 Sassari, Italia}
\email{steger@ssmain.uniss.it}

\thanks{This research was supported by the Australian Research Council.}

\usepackage{amsmath} 
\usepackage{amscd}

\input{prepictex} 
\input{pictex} 
\input{postpictex} 
\chardef\bslash=`\\ 
 
\makeatletter 
\def\verbatim{\interlinepenalty\@M \@verbatim 
  \leftskip\@totalleftmargin\advance\leftskip2pc 
  \frenchspacing\@vobeyspaces \@xverbatim} 
\makeatother 
\newtheorem{theorem}{Theorem}[section] 
 
\newtheorem{lemma}[theorem]{Lemma}

\theoremstyle{definition}
\newtheorem{example}[theorem]{Example}
 
\newtheorem{remark}[theorem]{Remark}

\newcounter{picture} 

\newsymbol\rtimes 226F 
 
\newcommand{\FF}{{\mathbb F}} 
 
\newcommand{\KK}{{\mathbb K}}

\newcommand{\ZZ}{{\mathbb Z}} 
 
\newcommand{\cB}{{\mathcal B}}

\newcommand{\cT}{{\mathcal T}} 
 
\newcommand{\G}{{\Gamma}}

\newcommand{\PGL}{{\text{\rm{PGL}}}} 
 
\newcommand{\wt}{\widetilde}

\newcommand{\pgl}{{\text{\rm{PGL}}}}

\begin{document} 

\def\Proof  {{\bf Proof.}\par}
\def\bs {\backslash}
\def\diam   {{\text{\rm diam}}}

\begin{abstract}
Let $\Gamma$ be a group of type rotating automorphisms of a building $\cB$ of type $\widetilde A_2$, 
and suppose that $\Gamma$ acts freely and transitively on the vertex set of $\cB$. The apartments of $\cB$ are tiled by triangles, labelled according to $\Gamma$-orbits. Associated with these tilings 
there is a natural subshift of finite type, which is shown to be irreducible. The key element in the proof is a combinatorial result about finite projective planes.
\end{abstract}

\maketitle

\section{Introduction}

Let $\cB$ be a locally finite thick affine building of type~$\wt A_2$ \cite{gar}.
Such a building $\cB$ is a two dimensional simplicial complex which is a union of  two dimensional subcomplexes,
called {\it apartments}.
The apartments  are Coxeter complexes of type $\widetilde A_2$, which may be realized as a tilings of the Euclidean plane by equilateral triangles. 
Buildings of type $\widetilde A_2$ \texttt{}are contractible as topological spaces and are natural two dimensional analogues of homogeneous trees. (A homogeneous tree is  a building of type $\widetilde A_1$.) 
Each vertex $v$ of $\cB$ is labeled with a {\it type} $\tau (v) \in \ZZ/3\ZZ$,
and each chamber has exactly one vertex of each type.
An automorphism $\alpha$ of $\cB$ is said to be {\it type rotating} if there exists $i \in 
\ZZ/3\ZZ$ such that $\tau(\alpha(v)) = \tau(v)+i$ for all vertices $v \in \cB$.

\refstepcounter{picture}
\begin{figure}[htbp]
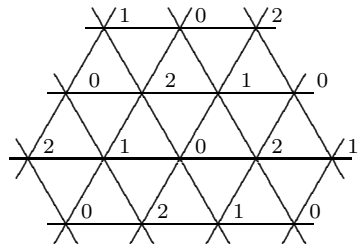
\label{fig5}
\centerline{
\beginpicture
\setcoordinatesystem units  <0.5cm, 0.866cm>        
\setplotarea x from -2.5 to 5, y from -1 to 2.5   
\put {$_1$} [l] at -1.6 2.2
\put {$_0$} [l] at  0.4  2.2
\put {$_2$} [l] at  2.4  2.2
\put {$_0$} [l] at  3.4  -0.8
\put {$_1$}  [l] at 1.4  -0.8
\put {$_2$} [l] at  -0.6  -0.8
\put {$_0$} [l] at -2.6  -0.8
\put {$_2$} [l]  at  -3.6 0.2
\put {$_1$}  [l]  at  -1.6 0.2
\put {$_1$} [l] at    4.4 0.2
\put {$_2$} [l] at    2.4 0.2
\put {$_0$}  [l] at 0.4  0.2
\put {$_0$} [l] at -2.4  1.2
\put {$_2$} [l] at -0.4 1.2
\put {$_1$} [l]  at  1.6  1.2
\put {$_0$} [l] at 3.6 1.2
\putrule from -2.5   2     to  2.5  2
\putrule from  -3.5 1  to 3.5 1
\putrule from -4.5 0  to 4.5 0
\putrule from -3.5 -1  to  3.5 -1
\setlinear
\plot  -4.3 -0.3  -1.7 2.3 /
\plot -3.3 -1.3  0.3 2.3 /
\plot -1.3 -1.3  2.3 2.3 /
\plot  0.7 -1.3  3.3 1.3 /
\plot  2.7 -1.3  4.3 0.3 /
\plot  -4.3 0.3  -2.7 -1.3 /
\plot  -3.3 1.3   -0.7 -1.3 /
\plot  -2.3 2.3   1.3 -1.3 /
\plot  -0.3 2.3   3.3 -1.3 /
\plot  1.7 2.3   4.3 -0.3 /
\endpicture
}
\caption{Part of an apartment in an $\tilde A_2$ building, showing vertex types.}
\end{figure}

If $\cB$ is a building of type $\widetilde A_2$ then the set $S_v$ of vertices of $\cB$ adjacent to any vertex
$v$ may be given the structure
of a finite projective plane. The projective planes corresponding to different vertices $v$ may
be nonisomorphic \cite{rt}, but they all have the same order $q$.
If a vertex $v$ of $\cB$ has type $i$ then the set $P$ of vertices of type $i+1$ in $S_v$
correspond to the $q^2+q+1$ points of the projective plane. The set $L$ of vertices of type $i+2$ in $S_v$
correspond to the $q^2+q+1$ lines of the projective plane. A point $p\in P$ and a line
$l\in L$ are incident in the projective plane if and only if there is an edge connecting them in the building.
The integer $q$ is called the order of the building and each edge in $\cB$ lies on $q+1$ triangles.
The reason for this is that every line in the projective plane is incident with $q+1$ points and every point is incident with $q+1$ lines. These facts will be used repeatedly below.

Suppose that $\cB$ is a building of type $\widetilde A_2$ and that  $\G$ is a group of type rotating automorphisms of $\cB$ which acts freely and transitively on the vertex set of $\cB$. Such groups $\G$ are called  $\widetilde A_2$ groups. In some ways, $\widetilde A_2$ groups are rank two analogues of finitely generated free groups, which act in a similar way on buildings of type $\widetilde A_1$ (trees).  The theory of $\widetilde A_2$ groups has been developed in detail in \cite{cmsz}. The $\widetilde A_2$ groups have a detailed combinatorial structure which makes them an ideal place to attack problems involving higher rank groups. 

An $\widetilde A_2$ group can be described as follows \cite [I,\S3]{cmsz}.
Let $(P,L)$ be a projective plane of order $q$. 
Let $\lambda : P \rightarrow L$ be a bijection (a {\it point--line
correspondence}).  Let ${\mathcal T}$
be a set of triples $(x, y, z)$ where $x, y, z \in P$, with the following
properties.

(i)  Given $x, y \in P$, then $(x, y, z) \in {\mathcal T}$ for some
$z \in P$ if and only if $y$ and $\lambda(x)$ are incident (i.e. $y \in 
\lambda(x)$).

(ii)  $(x, y, z) \in {\mathcal T} \Rightarrow (y, z, x) \in {\mathcal T}$.

(iii)  Given $x, y \in P$, then $(x, y, z) \in {\mathcal T}$ for at
most one $z \in P$.

${\mathcal T}$ is called a {\it triangle presentation}  compatible with
$\lambda$.  A complete list is given in \cite{cmsz} of all triangle
presentations for $q = 2$
and $q = 3$.

Let $\{a_x : x \in P\}$ be  $q^2 + q + 1$ distinct letters and form the
group
$$
\Gamma = \big\langle a_x, x \in P \ |\  a_x a_y a_z = 1 \hbox { for } (x,
y, z) \in {\mathcal T}
\big \rangle $$

 The Cayley  graph of $\Gamma$ with respect to the
generators $a_x, x \in P$ is the $1$-skeleton of an 
affine building of type $\widetilde A_2$.
It is convenient to identify the point $x \in P$ with the generator
$a_x \in \Gamma$. If $x\in P$ then the line $\lambda(x)$ corresponds to the inverse
$a_x^{-1}$ \cite{cmsz}. We therefore write $x^{-1}$ for $a_x^{-1}$ and identify $x^{-1}$ with
$\lambda(x)$. From now on the notation $x$ and
$\lambda(x)$ is used to represent $a_x$ and $a_{\lambda (x)}$ respectively. Note that, with this notation,
$${\mathcal T} = \{ (x,y,z) : x,y,z \in P  \hbox { and }   \ xyz = 1 \}.$$
This means that if $x,y \in P$ then $y \in \lambda(x)$ if and only if 
$xyz = 1$ for some $z \in P$.

The Cayley graph of $\Gamma$ will be regarded as a directed graph.
Vertices are identified with elements of $\G$ and a directed edge of the form $(a,as)$ with $a\in\G$ is labeled by a generator $s\in P$.  
Figure \ref{A4} illustrates a typical triangle based at a vertex $a\in\cB$. 
\refstepcounter{picture}
\begin{figure}[htbp]
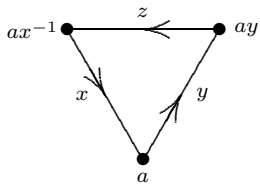
\label{A4}
\hfil
\centerline{
\beginpicture
\setcoordinatesystem units <1cm, 1.732cm>
\setplotarea  x from -2 to 2,  y from -0.1 to 1
\put {$\bullet$} at 0 0
\multiput {$\bullet$} at -1 1 *1 2 0 /
\put {$_a$} [t] at 0 -0.1
\put {$_{ax^{-1}}$} [r] at -1.1 1 
\put {$_{ay}$} [l] at 1.2 1
\arrow <10pt> [.2, .67] from  0.2 1 to 0 1
\arrow <10pt> [.2, .67] from  -0.7 0.7 to -0.5 0.5
\arrow <10pt> [.2, .67] from  0.3 0.3 to 0.5 0.5
\put {$_x$} [r ] at -0.7 0.5
\put {$_y$} [ l] at 0.7 0.5
\put {$_z$} [ b] at 0 1.1
\putrule from -1 1 to 1 1
\setlinear \plot -1 1 0 0 1 1 /
\endpicture.
}
\caption{A chamber based at a vertex $a$.}
\end{figure}

If $q=2$ there are eight $\widetilde A_2$ groups $\G$, all of which embed as lattices in the linear group $\PGL (3,\FF)$ over a local field $\FF$. If $q=3$ there are 89 possible $\widetilde A_2$ groups, of which 65  have buildings which are not associated with linear groups \cite{cmsz}.

\begin{example}\label{C1} The group C.1 of \cite{cmsz} has presentation
$$\langle
x_i, 0\le i \le 6\,
|\,
x_0x_0x_6,
x_0x_2x_3,
x_1x_2x_6,
x_1x_3x_5,
x_1x_5x_4,
x_2x_4x_5,
x_3x_4x_6
\rangle.$$
For this group, $q=2$, and there are $q^2+q+1=7$ generators.
Thus $P=\{x_0, \dots , x_6\}$ and $L=\{x_0^{-1}, \dots , x_6^{-1}\}$.
\end{example}

Two triangles lie in the same $\Gamma$-orbit if and only if they have the same edge labels, where each edge label is a generator of $\G$.
 The combinatorics of the finite projective plane $(P,L)$ shows that there are precisely $(q+1)(q^2+q+1)$ such labellings,
  which we refer to as {\it $\widetilde A_2$ triangle labellings}.
  Triangle labellings are in bijective correspondence with the elements of the triangle presentation $\cT$.
In Figure \ref{A6} we illustrate a triangle labelling (one of three) corresponding to the second relation in Example \ref{C1}.

\refstepcounter{picture}
\begin{figure}[htbp]
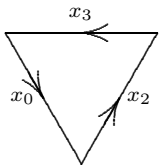
\label{A6}
\hfil
\centerline{
\beginpicture
\setcoordinatesystem units <1cm, 1.732cm>
\setplotarea  x from -2 to 2,  y from 0 to 1.2
\put {$_{x_3}$} [b ] at 0 1.1
\put {$_{x_2}$} [l ] at 0.6  0.5
\put {$_{x_0}$} [ r] at -0.6  0.5
\arrow <10pt> [.2, .67] from  0.2 1 to 0 1
\arrow <10pt> [.2, .67] from  -0.7 0.7 to -0.5 0.5
\arrow <10pt> [.2, .67] from  0.3 0.3 to 0.5 0.5
\putrule from -1 1 to 1 1 
\setlinear \plot -1 1 0 0 1 1  / 
\endpicture
}
\caption{A triangle labelling for the group C.1.}
\end{figure}

The edge labels (or equivalently the tiles) induce a tiling of the apartments in $\cB$,
as illustrated in Figure \ref{tiling}.

\refstepcounter{picture}
\begin{figure}[htbp]\label{tiling}
\hfil
\centerline{
\beginpicture
\setcoordinatesystem units <1cm, 1.732cm>
\setplotarea  x from -2 to 2,  y from -1 to 1.5
\put {$_{x_0}$} [r ] at -0.6 -0.5
\put {$_{x_2}$} [ l] at 0.6 -0.5
\put {$_{x_1}$} [r ] at 0.4  0.5
\put {$_{x_5}$} [ l] at 1.6  0.5
\put {$_{x_4}$} [r ] at -1.5  0.4
\put {$_{x_5}$} [ l] at -0.4  0.5
\put {$_{x_3}$} [t] at  0.0   -0.1
\put {$_{x_2}$} [b] at  -1.0  0.8
\put {$_{x_4}$} [b] at  1.0   0.8
\multiput {\beginpicture
\setcoordinatesystem units <1cm, 1.732cm>
\arrow <10pt> [.2, .67] from  0.2 1 to 0 1
\arrow <10pt> [.2, .67] from  -0.7 0.7 to -0.5 0.5
\arrow <10pt> [.2, .67] from  0.3 0.3 to 0.5 0.5
\putrule from -1 1 to 1 1 
\setlinear \plot -1 1 0 0 1 1 /
\endpicture}  at   0 0  1 1  -1 1  /
\endpicture
}
\caption{}
\end{figure}

There is a natural $\ZZ^2$ action on the space of
tiled apartments, which gives rise to a so called 2-dimensional subshift of finite type.

Consider the set of all apartments of $\cB$, with each triangle labelled as above.
Two matrices $M_1$, $M_2$ with entries in $\{0, 1\}$ are defined as follows.
If $\alpha, \beta \in \cT$, say that $M_1(\alpha, \beta)=1$ if and only if
the triangle labellings $\alpha=(a_1,a_2,a_3)$ and $\beta=(b_1,b_2,b_3)$ lie as shown on the right
of Figure \ref{XX1}. A similar definition applies for $M_2(\alpha,\gamma)=1$, as on the left of Figure \ref{XX1}.

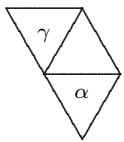
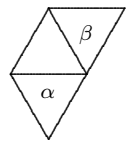
\refstepcounter{picture}
\begin{figure}[htbp]\label{XX1}
\hfil
\centerline{
\beginpicture
\setcoordinatesystem units <0.5cm,0.866cm>  point at -6 0 
\setplotarea x from -5 to 5, y from -1 to 2         
\put{$_{\beta}$}   at  1  1.6
\put{$_{\alpha}$}   at  0 0.7
\put{$M_1(\alpha, \beta)=1$}   at  0 -1
\setlinear \plot -1 1  0 0  1 1  /
\setlinear \plot -1 1  0 2  1 1  -1 1 /
\setlinear \plot 1 1  2 2   0 2  /
\setcoordinatesystem units <0.5cm,0.866cm>  point at 6 0 
\setplotarea x from -5 to 5, y from -1 to 2         
\put{$_{\gamma}$}   at  -1  1.6
\put{$_{\alpha}$}   at  0 0.7
\put{$M_2(\alpha, \gamma)=1$}   at  0 -1
\setlinear \plot -1 1  0 0  1 1  /
\setlinear \plot -1 1  0 2  1 1  -1 1 /
\setlinear \plot -1 1  -2 2   0 2  /
\endpicture
}
\hfil
\caption{Definition of the transition matrices.}
\end{figure}

The commuting matrices $M_1$, $M_2$ are the transition matrices associated with a 2-dimensional subshift, with alphabet $\cT$.
This subshift is said to be irreducible if for all $\alpha, \beta \in \cT$, there exist integers $r, s > 0$ such that
 the $(\alpha, \beta)$ component of the matrix $M_1^rM_2^s$ satisfies
 $$(M_1^rM_2^s)(\alpha,\beta)>0.$$
 A geometric interpretation of this condition is that any two triangle labellings $\alpha,\beta\in \cT$ can be realized
 so that $\beta$ lies in some sector with base labelled triangle $\alpha$, as in Figure \ref{XX2}.

\refstepcounter{picture}
\begin{figure}[htbp]
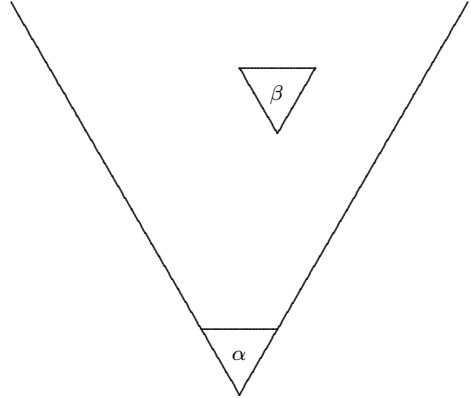
\label{XX2}
\hfil
\centerline{
\beginpicture
\setcoordinatesystem units <0.5cm,0.866cm>    
\setplotarea x from -6 to 6, y from 0.5 to 6         
\put {$_\alpha$}   at    0  0.6
\put {$_\beta$}    at    1  4.6
\setlinear
\plot -6  6   0 0   6 6 /
\plot -1  1    1 1 /
\plot 1 4  2 5  0 5   1 4  /
\endpicture
}
\hfil
\caption{The condition for irreducibility.}
\end{figure}

It is important for the simplicity of the $C^*$-algebras considered in \cite{rs}
that this subshift is irreducible.
In this article we prove irreducibility by showing that we can actually choose $r>0$ such that
$M_1^r(\alpha,\beta)>0.$
Thus $\beta$ lies on the wall of a sector as in Figure \ref{XX3}.  A similar
statement is true for the matrix $M_2$.

\refstepcounter{picture}
\begin{figure}[htbp]
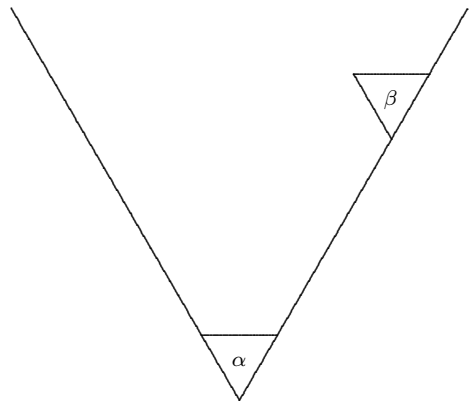
\label{XX3}
\hfil
\centerline{
\beginpicture
\setcoordinatesystem units <0.5cm,0.866cm>    
\setplotarea x from -6 to 6, y from 0.5 to 6         
\put {$_\alpha$}   at    0  0.6
\put {$_\beta$}    at    4  4.6
\setlinear
\plot -6  6   0 0   6 6 /
\plot -1  1    1 1 /
\plot 4 4  3 5  5 5  /
\endpicture
}
\hfil
\caption{Shifting along the wall of a sector.}
\end{figure}

Another way of viewing this is to say that irreducibility is proved for the one dimensional subshift associated with
tilings of strips between parallel walls in apartments, as illustrated in Figure \ref{strip}.
This is considerably stronger than irreducibility of the 2-dimensional subshift.

\refstepcounter{picture}
\begin{figure}[htbp]
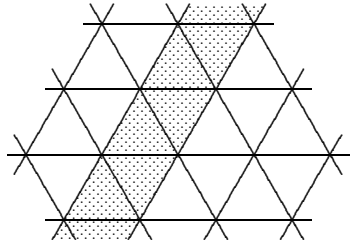
\label{strip}
\centerline{
\beginpicture
\setcoordinatesystem units  <0.5cm, 0.866cm>        
\setplotarea x from -3.5 to 5, y from -1.3 to 2.3   
\putrule from -2.5   2     to  2.5  2
\putrule from  -3.5 1  to 3.5 1
\putrule from -4.5 0  to 4.5 0
\putrule from -3.5 -1  to  3.5 -1
\setlinear
\plot  -4.3 -0.3  -1.7 2.3 /
\plot -3.3 -1.3  0.3 2.3 /
\plot -1.3 -1.3  2.3 2.3 /
\plot  0.7 -1.3  3.3 1.3 /
\plot  2.7 -1.3  4.3 0.3 /
\plot  -4.3 0.3  -2.7 -1.3 /
\plot  -3.3 1.3   -0.7 -1.3 /
\plot  -2.3 2.3   1.3 -1.3 /
\plot  -0.3 2.3   3.3 -1.3 /
\plot  1.7 2.3   4.3 -0.3 /
\setshadegrid span <1.5pt>
\vshade  -3.3  -1.3  -1.3  <,z,,>  -1.3  -1.3  0.7  <z,,,>  0.3   0.3  2.3  <z,,,>  2.3   2.3 2.3 /
\endpicture
}
\caption{A strip in an apartment.}
\end{figure}

Let $\Gamma$ be an $\widetilde A_2$ group. If $\Gamma$ has the property that the 2-dimensional subshift
described above is irreducible, then the theory developed in \cite[Section 7]{rs} applies.
This means that one may construct an associated simple $C^*$-algebra whose structure
was analyzed in \cite{rs}. The required irreducibility result was proved in \cite[Theorem 7.10]{rs}
only for the case where $\Gamma$ is a lattice in $\pgl_3(\KK)$, where $\KK$ is a local field
of characteristic zero. The argument
of \cite[Theorem 7.10]{rs} does not apply if $\cB$ is the building of $\pgl_3(\KK)$, where $\KK$ is a local field of positive characteristic, which is the case for the group C.1 of Example \ref{C1}.
Neither does it apply to many examples constructed in \cite{cmsz}, for which $\cB$ is not the Bruhat-Tits building of a linear group.
The purpose of the present article is to show that irreducibility holds for all $\widetilde A_2$ groups. This means that the theory of \cite{rs}
now applies to any such group.

\begin{remark}
The subshift studied in \cite{rs} was defined in terms of labelled parallelograms formed by a union of
two labelled triangles of the following form.

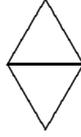
\refstepcounter{picture}
\begin{figure}[htbp]\label{modeltile}
\hfil
\centerline{
\beginpicture
\setcoordinatesystem units <0.5cm,0.866cm>   
\setplotarea x from -5 to 5, y from 0.5 to 2        
\putrule from -1 1 to 1 1
\setlinear \plot -1 1  0 0  1 1  /
\setlinear \plot -1 1  0 2  1 1 /
\endpicture
}
\hfil
\end{figure}
However, irreducibility of that subshift is an easy consequence of
the result presented here.
\end{remark}

We now state our main result.

\medskip

\begin{theorem}\label{main}
Given any two $\widetilde A_2$ triangle labellings, these labellings can be realized as the initial and final triangles of a sequence of triangles arranged along some wall in $\cB$ as follows:
\refstepcounter{picture}
\begin{figure}[htbp]
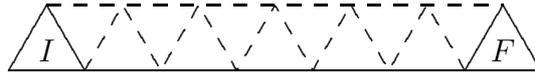
\label{fig1}
\centerline{
\beginpicture
\setcoordinatesystem units <0.5cm, 0.866cm >
\setplotarea  x from -6 to 4,  y from -1 to 0
\put {$I$} at -7 -0.7 
\put {$F$} at  5 -0.7  
\putrule from  -8 -1  to  6 -1 
\setlinear 
\plot -8 -1  -7 0   -6 -1  /
\plot  4 -1   5 0   6 -1 /
\setdashes 
\putrule from  -7 0  to  5 0
\plot
-6 -1  -5 0  -4 -1  -3  0  -2 -1  -1 0  0 -1  1 0  2 -1  3 0  4 -1 /
\endpicture
}
\hfil
\caption{Labelled triangles along a wall}
\end{figure}
\end{theorem}

The rest of the article is devoted to the proof of Theorem \ref{main}.

\bigskip

\section{Proof of Irreducibility of the 1-dimensional subshift}

Fix once and for all the triangle labellings $I$ and $F$.
Consider a triangle labelling of the form below (which we refer to as $\underset{b}{\Delta}$).

\centerline{
\beginpicture
\setcoordinatesystem units  <1cm, 1.732cm>
\setplotarea x from -1 to 1, y from -2 to 0.1         
\arrow <10pt> [.2, .67] from  0 -1  to  0.2 -1 
\arrow <10pt> [.2, .67] from   -0.5 -0.5 to -0.7 -0.7 
\arrow <10pt> [.2, .67] from   0.5 -0.5 to 0.3 -0.3
\put {$_{b_3}$} [r,b] at -0.7 -0.5
\put {$_{b_2}$} [l,b] at 0.7 -0.5
\put {$_{b}$} [ t] at 0 -1.1
\put {A triangle labelling of the form $\underset{b}{\Delta}$.} at 0 -1.6
\putrule from -1 -1 to 1 -1
\setlinear \plot -1 -1 0 0 1 -1 /
\endpicture
}

Call such a labelling  $\underset{b}{\Delta}$ {\it reachable} from the {\it left} if it is the final triangle labelling in some sequence with initial triangle $I$.  

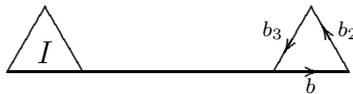
\centerline{
\beginpicture
\setcoordinatesystem units <0.5cm, 0.866cm >
\setplotarea  x from -8 to 1,  y from -2 to 0 
\putrule from  -8 -1  to  1 -1 
\setlinear 
\plot -8 -1  -7 0   -6 -1  /
\arrow <5pt> [.2, .67] from  0 -1  to  0.2 -1 
\arrow <5pt> [.2, .67] from   -0.5 -0.5 to -0.7 -0.7 
\arrow <5pt> [.2, .67] from   0.5 -0.5 to 0.3 -0.3
\put {$_{b_3}$} [r,b] at -0.7 -0.5
\put {$_{b_2}$} [l,b] at 0.7 -0.5
\put {$_{b}$} [ t] at 0 -1.1
\put {$I$} at -7 -0.7
\setlinear \plot -1 -1 0 0 1 -1 /
\endpicture
}

Similarly define {\it reachable} from the {\it right}.
\begin{figure}[htbp]\label{fig3r}
\centerline{
\beginpicture
\setcoordinatesystem units <0.5cm, 0.866cm >
\setplotarea  x from -1 to 8,  y from -1 to 0
\putrule from  -1 -1  to  8 -1 
\setlinear 
\plot  6 -1 7 0 8 -1 /
\arrow <5pt> [.2, .67] from  0 -1  to  0.2 -1 
\arrow <5pt> [.2, .67] from   -0.5 -0.5 to -0.7 -0.7 
\arrow <5pt> [.2, .67] from   0.5 -0.5 to 0.3 -0.3
\put {$_{b_3}$} [r,b] at -0.7 -0.5
\put {$_{b_2}$} [l,b] at 0.7 -0.5
\put {$_{b}$} [ t] at 0 -1.1
\put {$F$} at  7 -0.7
\setlinear \plot -1 -1 0 0 1 -1 /
\endpicture
}
\end{figure}

Note that for each edge labelling $b$ there are $q+1$ triangles of the form $\underset{b}{\Delta}$.
Therefore if we can show that there exists $b$ such that $\underset{b}{\Delta}$ is reachable from the left for more than $(q+1)/2$ values of the pair $(b_2,b_3)$ and reachable from the right for more that $(q+1)/2$ values of $(b_2,b_3)$, then there exists a labelling $(b,b_2,b_3)$ which is reachable both ways. This will prove Theorem \ref{main}.
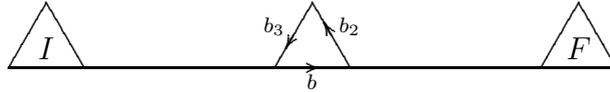
\begin{figure}[htbp]\label{fig3*}
\centerline{
\beginpicture
\setcoordinatesystem units <0.5cm, 0.866cm >
\setplotarea  x from -8 to 8,  y from -1 to 0
\putrule from  -8 -1  to  8 -1 
\setlinear 
\plot -8 -1  -7 0   -6 -1  /
\plot  6 -1 7 0 8 -1 /
\arrow <5pt> [.2, .67] from  0 -1  to  0.2 -1 
\arrow <5pt> [.2, .67] from   -0.5 -0.5 to -0.7 -0.7 
\arrow <5pt> [.2, .67] from   0.5 -0.5 to 0.3 -0.3
\put {$_{b_3}$} [r,b] at -0.7 -0.5
\put {$_{b_2}$} [l,b] at 0.7 -0.5
\put {$_{b}$} [ t] at 0 -1.1
\put {$I$} at -7 -0.7 
\put {$F$} at  7 -0.7 
\setlinear \plot -1 -1 0 0 1 -1 /
\endpicture
}
\end{figure}

In subsequent arguments, we will need to use a criterion for a triangle labelling of the form $\underset{c}{\Delta}$
to be reachable in one step from a triangle labelling of the form $\underset{b}{\Delta}$, as in Figure \ref{newfig}.

\refstepcounter{picture}
\begin{figure}[htbp]\label{newfig}
\hfil
\centerline{
\beginpicture
\setcoordinatesystem units  <1cm, 1.732cm>  
\setplotarea x from 1 to 7, y from -1.2 to 0.1      
\arrow <10pt> [.2, .67] from  2 -1  to  2.2 -1
\arrow <10pt> [.2, .67] from   1.5 -0.5 to 1.3 -0.7
\arrow <10pt> [.2, .67] from   2.5 -0.5 to 2.3 -0.3
\putrule from 1 -1 to 3 -1
\setlinear \plot 1 -1 2 0 3 -1 /
\setplotarea x from -1 to 1, y from -1.2 to 0.1         
\arrow <10pt> [.2, .67] from  4 -1  to  4.2 -1
\arrow <10pt> [.2, .67] from   3.5 -0.5 to 3.3 -0.7
\arrow <10pt> [.2, .67] from   4.5 -0.5 to 4.3 -0.3
\put {$_{b}$} [t] at 2 -1.1
\put {$_{c}$} [t] at 4 -1.1
\put {$_{\bullet}$} at 3 -1
\put {$_{\bullet}$} at 1 -1
\put {$_{\bullet}$} at 5 -1
\put {$_{\bullet}$} at  2 0
\put {$_{\bullet}$} at  4 0
\putrule from 3 -1 to 5 -1
\setlinear \plot 3 -1 4 0 5 -1 /
\setdashes
\putrule from 2 0 to 4 0
\endpicture
}
\caption{}
\end{figure}

\begin{lemma}\label{added}
Figure \ref{newfig} is possible in an apartment of $\cB$
if and only if $c\not\in \lambda(b)$.
\end{lemma}

\begin{proof}
Fix a vertex $v\in \cB$.  Since the 1-skeleton of $\cB$ is the Cayley graph of $(\Gamma, P)$, the vertex $v$ may be considered
as an element of $\Gamma$.
The choice of $v$ is irrelevant, by transitivity of the action of $\Gamma$.

As explained in the introduction, the set $S_v$ of vertices adjacent to $v$ has the structure of a finite projective plane.
The points of this projective plane are $\{ vx \, ; \, x\in P \}$ and the lines are  $\{ v\lambda(x) \, ; \, x\in P \}$.
Recall that $\lambda(x)=x^{-1}$ in the group $\Gamma$. Figure \ref{newfig} is therefore equivalent to
Figure \ref{newfigA}.
\refstepcounter{picture}
\begin{figure}[htbp]\label{newfigA}
\hfil
\centerline{
\beginpicture
\setcoordinatesystem units  <1cm, 1.732cm>  
\setplotarea x from 1 to 7, y from -1.2 to 0.1      
\arrow <10pt> [.2, .67] from  2 -1  to  2.2 -1
\arrow <10pt> [.2, .67] from   1.5 -0.5 to 1.3 -0.7
\arrow <10pt> [.2, .67] from   2.5 -0.5 to 2.3 -0.3
\putrule from 1 -1 to 3 -1
\setlinear \plot 1 -1 2 0 3 -1 /
\setplotarea x from -1 to 1, y from -1.2 to 0.1         
\arrow <10pt> [.2, .67] from  4 -1  to  4.2 -1
\arrow <10pt> [.2, .67] from   3.5 -0.5 to 3.3 -0.7
\arrow <10pt> [.2, .67] from   4.5 -0.5 to 4.3 -0.3
\put {$_{v\lambda(b)}$} [r] at 0.9 -1.1
\put {$_{vc}$} [l] at 5.1 -1.1
\put {$_{v}$} [t] at 3 -1.1
\put {$_{\bullet}$} at 3 -1
\put {$_{\bullet}$} at 1 -1
\put {$_{\bullet}$} at 5 -1
\put {$_{\bullet}$} at  2 0
\put {$_{\bullet}$} at  4 0
\putrule from 3 -1 to 5 -1
\setlinear \plot 3 -1 4 0 5 -1 /
\setdashes
\putrule from 2 0 to 4 0
\endpicture
}
\caption{}
\end{figure}

If $c\in \lambda(b)$, then there is an edge in $\cB$ between $v\lambda(b)$ and $vc$.
Figure \ref{newfigA} is therefore impossible, by contractibility of the building $\cB$.

On the other hand, if $c\not\in \lambda(b)$ then $v\lambda(b)$ and $vc$ are not adjacent in
$S_v$. Now  $v\lambda(b)$ and $vc$ lie in a hexagon $H$ whose vertices belong to $S_v$.
This is because the projective plane $S_v$ has the structure of a spherical building,
whose apartments are hexagons. The vertices of the hexagon $H$ are alternately points and lines
of the projective plane $S_v$. The only way in which the line $v\lambda(b)$ and the point
$vc$ can fail to be adjacent in the hexagon $H$ is if they are opposite vertices of
the hexagon, as shown in Figure \ref{hexagon}.

\refstepcounter{picture}
\begin{figure}[htbp]\label{hexagon}
\hfil
\centerline{
\beginpicture
\setcoordinatesystem units <0.5cm, 0.866cm>
\setplotarea  x from -2 to 3,  y from -1 to 1
\setlinear \plot   1 1  -1 1  -2 0  -1 -1   1 -1   2 0  1 1 /
\put {$_{\bullet}$} at -2 0
\put {$_{\bullet}$} at 2 0
\put {$_{v\lambda(b)}$} [r] at -2.4  0
\put {$_{vc}$} [l] at 2.3 0
\endpicture
}
\caption{}
\end{figure}

This means that Figure \ref{newfigA} is possible in $\cB$, where each labelled triangle has one edge on the hexagon $H$.

\end{proof}

\begin{lemma}\label{red} If $b\in P$ then the numbers
$${\mathcal L}(b) = \#\{(b_2,b_3): (b,b_2,b_3)\text{ is reachable from the left}\},$$
$${\mathcal R}(b) = \#\{(b_2,b_3): (b,b_2,b_3)\text{ is reachable from the right}\}$$
are {\it independent of $b$}.
\end{lemma}

\begin{proof}
It is clearly enough to prove the assertion for ${\mathcal L}(b)$.
Given  $b'\in P$, we must show that ${\mathcal L}(b)= {\mathcal L}(b')$.
Now the diagram in Figure \ref{fig4} can be completed by choosing $c$ such that $c\not\in\lambda(b)$ and $b'\not\in\lambda(c)$.  This is possible, since there exist  $q+1$ elements $c\in\lambda(b)$, there exist $q+1$ elements $c$ such that $b'\in\lambda(c)$, and $2(q+1)<q^2+q+1=\#(P)$.

\refstepcounter{picture}
\begin{figure}[htbp]\label{fig4}
\centerline{
\beginpicture
\setcoordinatesystem units <0.5cm, 0.866cm >
\setplotarea  x from -8 to 8,  y from -1 to 0
\putrule from  -8 -1  to  8 -1
\setlinear
\plot -8 -1  -7 0   -6 -1  /
\plot  6 -1 7 0 8 -1 /
\put {$_{b}$} [ t] at 3 -1.1
\put {$_{c}$} [ t] at 5 -1.1
\put {$_{b'}$} [ t] at 7 -1.1
\put {$I$} at -7 -0.7 
\setlinear 
\plot 2 -1   3 0  4 -1  5 0  6 -1  7 0  8 -1 /
\endpicture
}
\caption{}
\end{figure}

Choose and fix such an element $c\in P$. Then each labelling of $\underset{b}{\Delta}$ uniquely determines the labelling of $\underset{b'}{\Delta}$, and vice versa. That is,  for fixed $b,c,b'$, the number of labellings of $\underset{b}{\Delta}$ is the same as the number of labellings of $\underset{b'}{\Delta}$. It follows that ${\mathcal L}(b)\le {\mathcal L}(b')$. By symmetry, 
${\mathcal L}(b)= {\mathcal L}(b')$.
\end{proof}

It follows from Lemma \ref{red} that, in order to prove Theorem \ref{main}, it is enough to find an elements $b_1, b_2\in P$ such that 
\begin{subequations}\label{enough}
\begin{eqnarray}
{\mathcal L}(b_1) > (q+1)/2\, , \label{enough1a}\\
{\mathcal R}(b_2) > (q+1)/2. \label{enough1b}
\end{eqnarray}
\end{subequations}
It is clearly enough to verify (\ref{enough1a}).

Given the initial triangle labelling $I$, denote by $D$ the set of all $d\in P$ for which Figure \ref{fig55} is possible.
Thus $D$ contains precisely $q$ elements. 
For each $d\in D$ let $S_d$ denote the set of $f\in P$ such that Figure \ref{fig55} is possible.
Therefore $\#(S_d)=q$.
\refstepcounter{picture}
\begin{figure}[htbp]\label{fig55}
\hfil
\centerline{
\beginpicture
\setcoordinatesystem units  <1cm, 1.732cm>
\setplotarea x from -1 to 1, y from -1.2 to 0.1         
\arrow <10pt> [.2, .67] from  0 -1  to  0.2 -1 
\arrow <10pt> [.2, .67] from   -0.5 -0.5 to -0.7 -0.7 
\arrow <10pt> [.2, .67] from   0.5 -0.5 to 0.3 -0.3
\putrule from -1 -1 to 1 -1
\setlinear \plot -1 -1 0 0 1 -1 /
\arrow <10pt> [.2, .67] from  2 -1  to  2.2 -1 
\arrow <10pt> [.2, .67] from   1.5 -0.5 to 1.3 -0.7 
\arrow <10pt> [.2, .67] from   2.5 -0.5 to 2.3 -0.3
\put {$I$} at 0 -0.7
\put {$_{d}$} [b,r] at 1.5 -0.4
\put {$_{f}$} [l,b] at 2.7 -0.5
\putrule from 1 -1 to 3 -1
\putrule from 0 0 to 2 0
\setlinear \plot 1 -1 2 0 3 -1 /
\endpicture
}
\caption{}
\end{figure} 

\begin{lemma}\label{technical} If $d_1, d_2\in D$ and $d_1\neq d_2$, then
$S_{d_1}\cap S_{d_2}$ contains at most one element. 
\end{lemma}

\begin{proof}
 If $f\in S_{d_1}\cap S_{d_2}$ then $d_1$, $d_2\in\lambda(f)$.  The two points $d_1$, $d_2$ in the projective plane determine the line $\lambda(f)$ uniquely.  That is, $f$ is uniquely determined.
\end{proof}

Let $S=\displaystyle \bigcup_{d\in D} S_d$. Then $S$ is the set of elements $f\in P$ such that a diagram like Figure \ref{fig55} is possible, for the given initial triangle $I$.
There are $q(q-1)/2$ sets of the form $S_{d_i}\cap S_{d_j}$, each of which contains at most one element.
It follows from the exclusion-inclusion principle that
\begin{equation}\label{recall}
\#(S)\geq q.q-\frac{q(q-1)}{2}=\frac{q^2+q}{2}\,.
\end{equation}
This gives a lower bound on the number of possible edge labels $f$ in Figure \ref{fig55}.
Now let $f\in S$ be such an edge label. Then $f\in S_d$ for some $d\in P$.  Consider diagrams of the form illustrated in Figure \ref{fig66}.

\refstepcounter{picture}
\begin{figure}[htbp]\label{fig66}
\hfil
\centerline{
\beginpicture
\setcoordinatesystem units  <1cm, 1.732cm>
\setplotarea x from -1 to 5, y from -1.2 to 0.1         
\arrow <10pt> [.2, .67] from  0 -1  to  0.2 -1 
\arrow <10pt> [.2, .67] from   -0.5 -0.5 to -0.7 -0.7 
\arrow <10pt> [.2, .67] from   0.5 -0.5 to 0.3 -0.3
\putrule from -1 -1 to 1 -1
\setlinear \plot -1 -1 0 0 1 -1 /
\arrow <10pt> [.2, .67] from  2 -1  to  2.2 -1 
\arrow <10pt> [.2, .67] from   1.5 -0.5 to 1.3 -0.7 
\arrow <10pt> [.2, .67] from   2.5 -0.5 to 2.3 -0.3
\put {$I$} at 0 -0.7
\put {$_{d}$} [r,b] at 1.3 -0.5
\put {$_{f}$} [l,b] at 2.6 -0.6
\putrule from 1 -1 to 3 -1
\putrule from 0 0 to 2 0
\setlinear \plot 1 -1 2 0 3 -1 /
\setplotarea x from -1 to 1, y from -1.2 to 0.1         
\arrow <10pt> [.2, .67] from  4 -1  to  4.2 -1 
\arrow <10pt> [.2, .67] from   3.5 -0.5 to 3.3 -0.7 
\arrow <10pt> [.2, .67] from   4.5 -0.5 to 4.3 -0.3
\put {$_{g}$} [l] at 3.6 -0.6
\put {$_{k}$} [l,b] at 4.7 -0.5
\put {$_{h}$} [t] at 4 -1.1
\put {$_{x}$} [t] at 3 -1.1
\put {$_{\bullet}$} at 3 -1
\putrule from 3 -1 to 5 -1
\setlinear \plot 3 -1 4 0 5 -1 /
\setdashes 
\putrule from 2 0 to 4 0
\endpicture
}
\caption{}
\end{figure}

In the projective plane of nearest neighbours of $x$ label the points $p_f$, $p_g$ and  lines $l_f$, $l_h$ as in Figure \ref{pp}. (By duality, the words `point' and `line' could be interchanged here. The specified choice makes the wording of a later argument easier.)

\refstepcounter{picture}
\begin{figure}[htbp]\label{pp}
\hfil
\centerline{
\beginpicture
\setcoordinatesystem units  <1cm, 1.732cm>  
\setplotarea x from 1 to 7, y from -1.2 to 0.1      
\arrow <10pt> [.2, .67] from  2 -1  to  2.2 -1
\arrow <10pt> [.2, .67] from   1.5 -0.5 to 1.3 -0.7
\arrow <10pt> [.2, .67] from   2.5 -0.5 to 2.3 -0.3
\put {$_{d}$} [r,b] at 1.3 -0.5
\put {$_{f}$} [l,b] at 2.6 -0.6
\putrule from 1 -1 to 3 -1
\setlinear \plot 1 -1 2 0 3 -1 /
\setplotarea x from -1 to 1, y from -1.2 to 0.1         
\arrow <10pt> [.2, .67] from  4 -1  to  4.2 -1
\arrow <10pt> [.2, .67] from   3.5 -0.5 to 3.3 -0.7
\arrow <10pt> [.2, .67] from   4.5 -0.5 to 4.3 -0.3
\put {$_{g}$} [l] at 3.6 -0.6
\put {$_{k}$} [l,b] at 4.7 -0.5
\put {$_{h}$} [t] at 4 -1.1
\put {$_{x}$} [t] at 3 -1.1
\put {$_{p_f}$} [r] at 0.9 -1.1
\put {$_{l_h}$} [l] at 5.1 -1.1
\put {$_{l_f}$} [r] at 1.9 0.1
\put {$_{p_g}$} [l] at 4.1 0.1
\put {$_{\bullet}$} at 3 -1
\put {$_{\bullet}$} at 1 -1
\put {$_{\bullet}$} at 5 -1
\put {$_{\bullet}$} at  2 0
\put {$_{\bullet}$} at  4 0
\putrule from 3 -1 to 5 -1
\setlinear \plot 3 -1 4 0 5 -1 /
\setdashes
\putrule from 2 0 to 4 0
\endpicture
}
\caption{}
\end{figure}

Then $(g,h,k)$ is reachable from $f$, (i.~e. ~the diagram is possible), if and only if
$l_h\neq l_f$, $p_g\neq p_f$ and $p_g\in l_h\cap l_f$.  That is $p_g = (l_f-\{p_f\})\cap l_h$,
where $l_h\neq l_f$.

For  $h\in P$ the set of possible $g$ is in bijective correspondence with the set
\begin{eqnarray*}
T_h&=&\bigcup_{f\in S-\{h\}}\{(l_f-\{p_f\})\cap l_h\}\\
&=& l_h\cap\bigcup_{f\in S-\{h\}}l_f-\{p_f\}\,.
\end{eqnarray*}

If we can show that $\#(T_h)>\frac{q+1}{2}$ for some $h$, then (\ref{enough1a}) is satisfied with $b_1=h$. 

The proof of Theorem \ref{main} therefore reduces to the following combinatorial result about projective planes.
Recall from (\ref{recall}) that $\#(S)\geq\frac{q^2+q}{2}$.  
\begin{lemma}
In a projective plane of order $q$,
let $\left\{l_j:1\leq j\leq \frac{q^2+q}{2}\right\}$ be a family of distinct lines.   For each $j$, let $p_j$ be a point on $l_j$ and let $l_j'=l_j-\{p_j\}$.
Then there exists a line $m$ such that
$$\#\left(m\cap \bigcup\{l_j':l_j\neq m\}\right)>\frac{q+1}{2}\,.$$
\end{lemma}

\begin{proof} 
This divides into 3 separate cases, which are dealt with in increasing order of difficulty. \\
{\bf Case 1:} $q=2$. Here  $\frac{q^2+q}{2}=3$, and so there are three distinct lines $l_1$, $l_2$, $l_3$, each containing three points.  Each set $l_j'$ therefore contains exactly two points.

Choose a line $m$ which meets a point of $l_1'-l_2$ and a point of $l_2'-l_1$.
Then $m\cap(l_1'\cup l_2')$ contains $2>\frac{3}{2}$ elements.
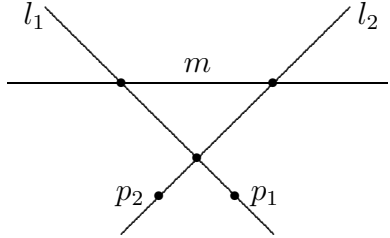
\begin{figure}[htbp]\label{spp}
\centerline{
\beginpicture
\setcoordinatesystem units <1cm, 0.5cm >
\setplotarea  x from -4 to 4,  y from -3 to 3
\putrule from  -2.5 1  to  2.5  1 
\setlinear 
\plot -1 -3    2  3  /
\plot  1 -3   -2  3  /
\put {$_{\bullet}$} at 0 -1
\put {$_{\bullet}$} at 0.5  -2  
\put {$_{\bullet}$} at -0.5 -2 
\put {$_{\bullet}$} at  1  1 
\put {$_{\bullet}$} at  -1 1 
\put {$m$}[b] at 0 1.3
\put {$l_1$}[r]  at -2  2.8  
\put {$l_2$}[l]  at 2.1  2.8 
\put {$p_1$}[l] at  0.7  -2 
\put {$p_2$}[r]  at  -0.7  -2 
\endpicture
}
\hfil
\end{figure}

{}\hfill

\noindent {\bf Case 2:} $q\geq 4$.\\
Each line contains $q+1$ points, so $\#(l_j')=q$.  Two distinct lines meet in exactly one point.  Hence
\begin{equation}\label{ineq1}
\#(l_1'\cup l_2'\cup l_3')\geq 3q-3\,.
\end{equation}
Assume that the conclusion of the Lemma is false.
Then we {\it claim} that for $3\leq k\leq\frac{q^2+q}{2}$,
\begin{equation}\label{ineq2}
\#(l_1'\cup l_2'\cup\cdots\cup l_k')\geq(3q-3)+(k-3)\lceil\frac{q-1}{2}
\rceil\,,
\end{equation}
where $\lceil t \rceil$ denotes the ceiling of $t$, the least integer not less than $t$.

We prove the claim by induction.
If $k=3$ then it is true, by (\ref{ineq1}).
Assume that (\ref{ineq2}) holds for a given value of $k$.
Since we are assuming that the conclusion of the Lemma fails,
$$\#(l_{k+1}'\cap(l_1'\cup\cdots \cup l_k'))\leq\#(l_{k+1}\cap(l_1'\cup\cdots\cup l_k'))\leq\frac{q+1}{2}\,.$$
Hence,
$$\#(l'_{k+1}-(l'_1\cup\cdots\cup l'_k))\geq q-\frac{q+1}{2}=\frac{q-1}{2}\,.$$
Therefore
\begin{eqnarray*}
\#(l_1'\cup l_2'\cup\cdots\cup l_k'\cup l_{k+1}')
    &\geq (3q-3)+(k-3)\lceil\frac{q-1}{2}\rceil+\lceil\frac{q-1}{2}\rceil \\
    &=(3q-3)+((k+1)-3)\lceil\frac{q-1}{2}\rceil\,.
\end{eqnarray*}
Thus we have established (\ref{ineq2}).

In particular, since (\ref{ineq2}) holds for $k=(q^2+q)/2$, and there are $q^2+q+1$ points in the projective plane, we have
\begin{equation}\label{ineq3}
q^2+q+1\geq(3q-3)+\left(\frac{q^2+q}{2}-3\right)\lceil\frac{q-1}{2}\rceil\,.
\end{equation}

\noindent Now (\ref{ineq3}) has been derived from the assumption that the conclusion of the Lemma was false. Therefore all that is required now is to show that (\ref{ineq3}) is false.
Now (\ref{ineq3}) fails when $q=4$, since in that case
$$q^2+q+1=21\not\geq 23=(3q-3)+\left(\frac{q^2+q}{2}-3\right)\lceil\frac{q-1}{2}\rceil\,.$$
On the other hand, if $q\geq 5$, write $q=r+5$, $r\geq 0$.
Then
\begin{eqnarray*}
4\left(3q-3+\left(\frac{q^2+q}{2}-3\right)\left(\frac{q-1}{2}\right)-(q^2+q+1)\right)
    &=&q^3-4q^2+q-10\\
    &=&r^3+11r^2+36r+20\\
    &\geq& 20\,.
\end{eqnarray*}
Therefore (\ref{ineq3}) also fails when $q\geq 5$.
This proves Case 2.

{}\hfill

\noindent {\bf Case 3:} $q=3$. This requires separate treatment.
Here $\frac{q^2+q}{2}=6$, $\frac{q+1}{2}=2$.

Given distinct lines $l_1, l_2 \ldots, l_6$ we delete a point from each to obtain sets $l_1',\ldots,l_6'$.  We must find a line $m$ such that
\begin{equation}\label{three}
\#\left(m\cap\bigcup\{l_j':l_j\neq m\}\right)>2\,.
\end{equation}
It is known that there is a unique projective plane of order 3, namely the Desarguesian plane arising from a 3-dimensional vector space over $\FF_3$ \cite[Theorem 2.3.1]{bl}.

There are thirteen points and thirteen lines in the projective plane. Label the points $0,1,2,\ldots,12$ and
label the lines $(0),(1),(2),\ldots,(12)$, as indicated in the table below \cite[Section 1.4]{bl}.
For example, line $(8)$ contains the points $5,6,8,1$.
\begin{table}[hbt]
\begin{center}
\begin{tabular}{||c|c|c|c|c|c|c|c|c|c|c|c|c||}
\hline
(12)&(11)&(10)&(9)&(8)&(7)&(6)&(5)&(4)&(3)&(2)&(1)&(0)\\ \hline
1&2&3&4&5&6&7&8&9&10&11&12&0\\
2&3&4&5&6&7&8&9&10&11&12&0&1\\
4&5&6&7&8&9&10&11&12&0&1&2&3\\
\underline{10}&\underline{11}&12&0&1&2&3&4&5&6&7&8&9\\ \hline
\end{tabular}
\end{center}
\end{table}

By permuting the lines $l_1, l_2 \ldots, l_6$, if necessary,
we may suppose that $l_1\cap l_2$ is not equal to either of the excluded points $p_1$ or $p_2$. 

To check this assertion, suppose that it does not already hold for the given choice of $l_1, l_2$.
Since each point is incident at most four of the lines $l_1, l_2 \ldots, l_6$, we may assume that $l_1\cap l_2=p_1$ but that $l_1\cap l_5\ne p_1$ and  $l_1\cap l_6\ne p_1$. If $l_1\cap l_5\ne p_5$  or $l_1\cap l_6\ne p_6$ we are done. On the other hand, if 
$l_1\cap l_5 = p_5$, $l_1\cap l_6 = p_6$ and $p_5\ne p_6$ then $l_5 \cap l_6$ is not equal to either $p_5$ or $p_6$. It remains to consider the case $l_1\cap l_5 = p_5$, $l_1\cap l_6 = p_6$ with $p_5=p_6$.
In that case, $l_2\cap l_5\ne p_5$, $l_2\cap l_6\ne p_6$ (since $l_1\cap l_2 =p_1$) and  either $l_2\cap l_5\ne p_2$ or $l_2\cap l_6\ne p_2$.

Having verified this assertion, we can assume that $l_1\cap l_2$,  $p_1$ and $p_2$ are three noncollinear points. Now the automorphism group $\pgl_3(\FF_3)$ acts transitively on triples of noncollinear points.
Map these three points to the points $2, 10, 11$ respectively
We may therefore suppose that $l_1$, $l_2$ are lines $(12), (11)$ respectively with excluded points $10, 11$ (underlined in the table).  Thus 
$$l_1'=\{1,2,4\}\qquad l_2'=\{2,3,5\}\,.$$
Now for $j=3,4,5,6$, the set $l_j'$ contains a point not in $l_1$ or $l_2$, namely one of the points 0,6,7,8,9,12. Let $j\in\{3,4,5,6\}$. 

\begin{itemize}
\item[(a)] If $0\in l_j'$ then
$$(0)\cap(l_1'\cup l_2'\cup l_j')=\{1,3,0\}\,,$$
and
$$(9)\cap(l_1'\cup l_2'\cup l_j')=\{4,5,0\}\,.$$
One can choose as line $m$ to satisfy inequality (\ref{three}) whichever of $(0)$, $(9)$ is not equal to $l_j$. Both choices of $m$ may be possible.\\
\item[(b)] If $6\in l_j'$ then 
$$(8)\cap(l_1'\cup l_2'\cup l_j')=\{1,5,6\}\,,$$
and
$$(10)\cap(l_1'\cup l_2'\cup l_j')=\{4,3,6\}\,.$$
One can choose as line $m$ whichever of $(8)$, $(10)$ is not equal to $l_j$.\\
\item[(c)] If $7\in l_j'$ then
$$(9)\cap(l_1'\cup l_2'\cup l_j')=\{4,5,7\}\,,$$
so if $l_j\neq(9)$  we can choose $m=(9)$.

\noindent If $8\in l_j'$ then 
$$(8)\cap(l_1'\cup l_2'\cup l_j')=\{1,5,8\}\,,$$
so if $l_j\neq(8)$ we can choose $m=(8)$.

\noindent If $9\in l_j'$ then 
$$(0)\cap(l_1'\cup l_2'\cup l_j')=\{1,3,9\}\,,$$
so if $l_j\neq(0)$ we can choose $m=(0)$.

\noindent If $12\in l_j'$ then 
$$(10)\cap(l_1'\cup l_2'\cup l_j')=\{4,3,12\}\,,$$
so if $l_j\neq(10)$ we can choose $m=(10)$.\\

\item[(d)] By choosing $j=3,4,5,6$ in parts (a), (b) and (c) above that,
we see that we can choose $m$ to satisfy inequality (\ref{three}) except in one case.
Up to a permutation of the set $\{3, 4, 5, 6\}$,
this is the case where 
$$l_3,\,l_4,\,l_5,\,l_6=(9),\,(8),\,(0),\,(10)$$
respectively with
$$7\in(9)',\, 8\in(8)',\, 9\in(0)',\, 12\in(10)'\,.$$

\noindent We work with the three lines $l_1=(12)$, $l_3=(9)$, $l_6=(10)$.
There are two possibilities to consider:

If $6\in (10)'$ then $(7)\cap(l_1'\cup l_3'\cup l_6')=\{2,7,6\}$; so take $m=(7)$.

If $6\not\in(10)'$ then $(10)'=\{3,4,12\}$, $(2)\cap(l_1'\cup l_3'\cup l_6')=\{1,7,12\}$; so take $m=(2)$.
\end{itemize}
\end{proof}

\medskip

\begin{remark}
Careful examination of the proof of Theorem \ref{main} shows that six steps are enough to get from initial to final triangle, exactly as indicated in Figure \ref{fig1}.
\end{remark}


\bigskip

\begin{thebibliography}{CMSZ}

\bibitem [Bl]{bl} J. ~W. ~Blattner, {\it Projective Plane Geometry}, Holden-Day,  San Francisco, 1968. 

\bibitem[CMSZ]{cmsz} D. I. Cartwright, A. M. Mantero, T. Steger and A. Zappa, Groups acting simply transitively on the vertices of a building of type $\widetilde A_2$, I,II,\ {\it Geom. Ded.}   {\bf 47} (1993), 143--166 and 167--223.

\bibitem[Gar]{gar} P.~Garrett, {\em Buildings and Classical Groups}, Chapman and Hall, London, 1997.

\bibitem[RS]{rs} G. Robertson and T. Steger, Affine buildings, tiling systems and higher rank Cuntz-Krieger algebras, {\it J. reine angew. Math.} {\bf 513} (1999), 115--144.

\bibitem[RT]{rt} M. ~A. ~Ronan and J. ~Tits, {Building buildings}, {\it Math. Ann.} {\bf 278} (1987),
291--306.

\end{thebibliography}
\end{document}